\documentclass[reqno]{amsart}
\usepackage{amsmath}
\usepackage{arydshln}%
\allowdisplaybreaks[3]
\numberwithin{equation}{section}
\numberwithin{equation}{section}

\newtheorem{thm}{\indent Theorem}[section]
\newtheorem{cor}[thm]{\indent Corollary}
\newtheorem{lem}[thm]{\indent Lemma}
\newtheorem{prop}[thm]{\indent Proposition}
\newtheorem{dfn}{{\indent\bf Definition}}[section]

\newtheorem{expl}{{\indent\bf Example}}[section]

\newcommand{\mb}{\mbox}

\newcommand{\dps}{\displaystyle}

\newcommand{\ol}{\overline}

\newcommand{\ttiny}{\fontsize{5pt}{\baselineskip}\selectfont}
\newcommand{\strl}[2]{\stackrel{\mbox{\ttiny $#1$}}{#2}}

\newcommand{\td}{\tilde}

\newcommand{\fr}{\frac}

\newcommand{\edd}{\end{document}}
\newcommand{\be}{\begin{equation}}
\newcommand{\ee}{\end{equation}}
\newcommand{\bsl}{\backslash}

\newcommand{\undbc}{\underbrace}

\newcommand{\lagl}{\langle}
\newcommand{\ragl}{\rangle}
\newcommand{\lmx}{\left(\begin{matrix}}
\newcommand{\rmx}{\end{matrix}\right)}
\newcommand{\ldt}{\left|\begin{matrix}}
\newcommand{\rdt}{\end{matrix}\right|}

\newcommand{\sgn}{{\rm Sgn\,}}

\newcommand{\tr}{{\rm tr\,}}

\newcommand{\veps}{\varepsilon}
\newcommand{\bbr}{{\mathbb R}}

\newcommand{\mo}{M\"obius }

\newcommand{\ba}{\begin{array}}
\newcommand{\ea}{\end{array}}
\newcommand{\nnm}{\nonumber}
\newcommand{\beal}{\begin{align}}
\newcommand{\eal}{\end{align}}
\newcommand{\bea}{\begin{eqnarray}}
\newcommand{\eea}{\end{eqnarray}}

\begin{document}

\title[Regular space-like hypersurfaces in $\mathbb{S}^{m+1}_{1}$ with parallel para-Blaschke tensors]{Regular space-like hypersurfaces in $\mathbb{S}^{m+1}_{1}$\\ with parallel para-Blaschke tensors}

\author[X. X. Li]{Xingxiao Li$^*$} 

\author[H. R. Song]{Hongru Song} 

\dedicatory{}

%
\thanks{Research supported by
Foundation of Natural Sciences of China (No. 11171091, 11371018).}

\address{Xingxiao Li
\endgraf School of Mathematics and Information Sciences
\endgraf Henan Normal University
\endgraf Xinxiang 453007, Henan,
P.R. China}%
\email{xxl$@$henannu.edu.cn}

\address{Hongru Song
\endgraf School of Mathematics and Information Sciences
\endgraf Henan Normal University
\endgraf Xinxiang 453007, Henan,
P.R. China}%
\email{yaozheng-shr$@$163.com}



\begin{abstract}
In this paper, we give a complete conformal classification of the regular space-like hypersurfaces in the de Sitter Space $\mathbb{S}^{m+1}_{1}$ with parallel para-Blaschke tensors.
\end{abstract}

\maketitle

\section{Introduction} 

Let $\bbr^{s+m}_s$ be the $(s+m)$-dimensional pseudo-Euclidean space which is the real vector space $\bbr^{s+m}$ equipped with the non-degenerate inner product $\lagl\cdot,\cdot\ragl_s$ given by
$$
\lagl \xi,\eta\ragl_s=-\xi_1\cdot \eta_1+\xi_2\cdot \eta_2,\quad
\xi=(\xi_1,\xi_2),\,\eta=(\eta_1,\eta_2)\in\bbr^s\times\bbr^m\equiv\bbr^{s+m}.
$$
where the dot ``$\cdot$'' is the standard Euclidean inner product either on $\bbr^{s}$ or $\bbr^{m}$.

Denote by $\mathbb{RP}^{m+2}$ the real projection space of dimension $m+2$. Then the so called conformal space $\mathbb{Q}^{m+1}_1$ is defined as (\cite{ag1})
$$
\mathbb{Q}^{m+1}_1={\{[\xi]\in \mathbb{RP}^{m+2};\ \xi\in\bbr^{m+3}_1, \lagl \xi,\xi\ragl_2=0\}},
$$
while, for any $a>0$, the standard sphere $\mathbb{S}^{m+1}(a)$, the hyperbolic space $\mathbb{H}^{m+1}\left(-\fr1{a^2}\right)$, the de Sitter space $\mathbb{S}^{m+1}_1(a)$ and the anti-de Sitter space $\mathbb{H}^{m+1}_1\left(-\fr1{a^2}\right)$ are defined accordingly by
\begin{align*}
\mathbb{S}^{m+1}(a)=&\{\xi\in\bbr^{m+2};\ \xi\cdot\xi=a^2\},\quad \mathbb{H}^{m+1}\left(-\fr1{a^2}\right)=\{\xi\in\bbr^{m+2}_1;\ \lagl \xi,\xi \ragl_1=-a^2\},\\
\mathbb{S}^{m+1}_1(a)=&\{\xi\in\bbr^{m+2}_1;\ \lagl \xi,\xi \ragl_1=a^2\},\quad \mathbb{H}^{m+1}_1\left(-\fr1{a^2}\right)=\{\xi\in\bbr^{m+2}_2;\ \lagl \xi,\xi \ragl_2=-a^2\}.
\end{align*}
Then $\mathbb{S}^{m+1}(a)$, $\mathbb{H}^{m+1}\left(-\fr1{a^2}\right)$ and the Euclidean space $\bbr^{m+1}$ are called Riemannian space forms, while $\mathbb{S}^{m+1}_1(a)$, $\mathbb{H}^{m+1}_1\left(-\fr1{a^2}\right)$ and the Lorentzian space $\bbr^{m+1}_1$ are called Lorentzian space forms. Denote
\begin{align*}
\mathbb{S}^{m+1}=&\mathbb{S}^{m+1}(1),\quad\mathbb{H}^{m+1}=\mathbb{H}^{m+1}\left(-1\right),\\
\mathbb{S}^{m+1}_1=&\mathbb{S}^{m+1}_1(1),\quad\mathbb{H}^{m+1}_1=\mathbb{H}^{m+1}_1\left(-1\right).
\end{align*}
Define three hyperplanes as follows:
\begin{align*}
&\pi=\{[\xi]\in \mathbb{Q}^{m+1}_1;\ \xi\in\bbr^{m+3}_1, \xi_1=\xi_{m+2}\},\\ &\pi_+=\{[\xi]\in \mathbb{Q}^{m+1}_1;\ \xi_{m+2}=0\},\\ &\pi_-=\{[\xi]\in \mathbb{Q}^{m+1}_1;\ \xi_1=0\}.
\end{align*}
Then there are three conformal diffeomorphisms of the Lorentzian space forms into the conformal space:
\begin{equation}\label{eq1.1}
\begin{aligned}
&\sigma_0:\bbr^{m+1}_1\to \mathbb{Q}^{m+1}_1\bsl\pi,\quad u\longmapsto\left[\left(\lagl u,u\ragl_1-1,2u,\lagl u,u\ragl_1+1\right)\right],\\
&\sigma_1:\mathbb{S}^{m+1}_1\to \mathbb{Q}^{m+1}_1\bsl\pi_+,\quad u\longmapsto\left[\left(u,1\right)\right],\\
&\sigma_{-1}:\mathbb{H}^{m+1}_1\to \mathbb{Q}^{m+1}_1\bsl\pi_-,\quad u\longmapsto\left[\left(1,u\right)\right].
\end{aligned}
\end{equation}
Therefore $\mathbb{Q}^{m+1}_1$ is the common conformal compactification of $\bbr^{m+1}_1$, $\mathbb{S}^{m+1}_1$ and $\mathbb{H}^{m+1}_1$.

As we know, the \mo geometry of umbilic-free submanifolds in the three Riemannian space forms, modeled on the standard unit sphere $\mathbb{S}^{m+1}$, has been studied extensively and a very ample bundle of interesting results in this area have been obtained ever since the pioneer work \cite{wcp} by C. P. Wang was published. Particularly, a lot of classification theorems have been proved in recent years, see for example the references \cite{clq}--\cite{lz2}. Note that due to the very recent achievement in \cite{L-W} and \cite{L-J-W}, the classification problems of both the \mo isoparametric and Blaschke isoparametric hypersurfaces have been solved completely.

On the other hand, same as the \mo geometry of submanifolds in the Riemannian space forms, the conformal geometry of submanifolds in the Lorentzian space forms is another important branch of conformal geometry, and  these two turn out to be closely related to each other. Note that Nie at al successfully set up a unified framework of conformal geometry for regular submanifolds in Lorentzian space forms by introducing the conformal space $\mathbb{Q}^{m+1}_1$ and the basic conformal invariants, that is, the conformal metric $g$, the conformal form $\Phi$, the Blaschke tensor $A$ and the conformal second fundamental form $B$. Under this framework, several characterization or classification theorems are obtained for regular hypersurfaces with some special conformal invariants, see for example (\cite{ag1}-- \cite{ag4}). The achievement of this kind, among others, certainly proves the efficiency of the above framework.

In a previous paper (\cite{LS}), we have used two conformal non-homogeneous coordinate systems on $\mathbb{Q}^{m+1}_1$, which are modeled on the de Sitter space $\mathbb{S}^{m+1}_1$ and denoted respectively by $\strl{(1)}{\Psi}$ and $\strl{(2)}{\Psi}$, to cover the conformal space $\mathbb{Q}^{m+1}_1$, so that the conformal geometry of regular space-like hypersurfaces in $\mathbb{Q}^{m+1}_1$ can be simply treated as that of regular space-like hypersurfaces in $\mathbb{S}^{m+1}_1$. It is easily seen that the same is true for general regular submanifolds. As a result in \cite{LS}, we have established a complete conformal classification of the regular space-like hypersurfaces in the de Sitter space $\mathbb{S}^{m+1}_1$ with parallel Blaschke tensors.

In this paper, we shall deal with the so called para-Blaschke tensor. In fact we are able to prove a more general theorem that gives a complete conformal classification for all the regular space-like hypersurfaces in $\mathbb{S}^{m+1}_1$ with parallel para-Blaschke tensors by proving first the vanishing of the conformal form $\Phi$ and then carefully analizing the distinct eigenvalues of the para-Blaschke tensor.

Note that, as shown in \cite{LS} (see also Section $2$ below), the above two conformal non-homogeneous coordinate maps $\strl{(1)}{\Psi}$ and $\strl{(2)}{\Psi}$ onto $\mathbb{S}^{m+1}_1$ are conformal equivalent to each other on where both of them are defined. Therefore, without loss of generality, we can simply use $\Psi$ to denote either one of $\strl{(1)}{\Psi}$ and $\strl{(2)}{\Psi}$. In this sense, the main theorem of the present paper is stated as follows:

\begin{thm}\label{main} Let $x:M^m\to \mathbb{S}^{m+1}_1$, $m\geq 2$, be a regular space-like hypersurface in the de Sitter space $\mathbb{S}^{m+1}_1$. Suppose that, for some constant $\lambda$, the corresponding para-Blaschke tensor $D^\lambda:=A+\lambda B$ of $x$ is parallel. Then $x$ is locally conformal equivalent to
one of the following hypersurfaces:

\begin{enumerate}

\item a regular space-like hypersurface in $\mathbb{S}^{m+1}_1$ with constant scalar curvature and constant mean curvature;

\item the image under $\Psi\circ\sigma_0$ of a regular space-like hypersurface in $\bbr^{m+1}_1$ with constant scalar curvature and constant mean curvature;

\item the image under $\Psi\circ\sigma_{-1}$ of a regular space-like hypersurface in $\mathbb{H}^{m+1}_1$
with constant scalar curvature and constant mean curvature;

\item $\mathbb{S}^{m-k}(a)\times\mathbb{H}^k\left(-\fr1{a^2-1}\right) \subset \mathbb{S}^{m+1}_1$, $a>1$, $k=1,\cdots m-1$;
\item the image under $\Psi\circ\sigma_0$ of $\mathbb{H}^k\left(-\fr1{a^2}\right)\times \bbr^{m-k} \subset \bbr^{m+1}_1$, $a>0$, $k=1,\cdots m-1$;
\item the image under $\Psi\circ\sigma_{-1}$ of $\mathbb{H}^k\left(-\fr1{a^2}\right)\times \mathbb{H}^{m-k}(-\fr1{1-a^2}) \subset \mathbb{H}^{m+1}_1$, $0<a<1$, $k=1,\cdots m-1$;
\item the image under $\Psi\circ\sigma_0$ of $WP(p,q,a)\subset \bbr^{m+1}_1$ for some constants $p,q,a$;

\item one of the regular space-like hypersurfaces as indicated in Example $3.2$;

\item one of the regular space-like hypersurfaces as indicated in Example $3.3$.
\end{enumerate}
\end{thm}

{\bf Remark} By using the same idea and similar argument in this paper, we can update and simplify the main theorem in \cite{clq} as follows:

\begin{thm}[cf. \cite{clq}] Let $x: M^m\to \mathbb{S}^{m+1}$, $m\geq 2$, be an umbilic-free hypersurface in the unit sphere $\mathbb{S}^{m+1}$. Suppose that, for some constant $\lambda$, the corresponding para-Blaschke tensor $D^\lambda:=A+\lambda B$ of $x$ is parallel. Then $x$ is  locally \mo equivalent to
one of the following hypersurfaces:

\begin{enumerate}

\item an immersed hypersurface in $\mathbb{S}^{m+1}$ with constant scalar curvature and constant mean curvature;

\item the image under $\sigma$, of an immersed hypersurface in $\bbr^{m+1}$ with constant scalar curvature and constant mean curvature;

\item the image under $\tau$, of an immersed hypersurface in $\mathbb{H}^{m+1}$ with constant scalar curvature and constant mean curvature;

\item a standard torus $\mathbb{S}^K(r)\times \mathbb{S}^{m-K}(\sqrt{1-r^2})$ in
$\mathbb{S}^{m+1}$ for some $r>0$ and positive integer $K$;

\item the image under $\sigma$, of a standard cylinder $\mathbb{S}^K(r)\times
\bbr^{m-K}$ in $\bbr^{m+1}$ for some $r>0$ and positive integer
$K$;

\item the image under $\tau$, of a standard cylinder $\mathbb{S}^K(r)\times
\mathbb{H}^{m-K}(-\fr1{1+r^2})$ in $\mathbb{H}^{m+1}$ for some $r>0$ and positive
integer $K$;

\item $CSS(p,q,r)$ for some constants $p,q,r$ (see Example $3.1$ in \cite{clq});

\item one of the immersed hypersurfaces as indicated in Example $3.2$ in \cite{clq};

\item one of the immersed hypersurfaces as indicated in Example $3.3$  in \cite{clq}.
\end{enumerate}
where the conformal embeddings $\sigma:\mathbb{R}^{m+1}\to \mathbb{S}^{m+1}$ and $\tau:\mathbb{H}^{m+1}\to \mathbb{S}^{m+1}$ are defined in (1.1) of \cite{clq}.
\end{thm}

\section{Necessary basics on regular space-like hypersurfaces}

This section provides some basics of the conformal geometry of regular space-like hypersurfaces in the Lorentzian space forms. The main idea comes originally from the work of C.P. Wang on the M\"obius geometry of umbilic-free submanifolds in the unit sphere (\cite{wcp}), and much of the details can be found in a series of papers by Nie at al (see for example \cite{ag1}--\cite{ag4}).

Let ${x}:M^m\to \mathbb{S}^{m+1}_{1}\subset\bbr^{m+2}_1$ be a regular space-like hypersurface in $\mathbb{S}^{m+1}_{1}$.
Denote by  ${h}$ the (scalar-valued) second fundamental form of ${x}$
with components ${h}_{ij}$ and ${H}=\fr1m\tr {h}$
the mean curvature. Define the conformal factor $\rho>0$ and the conformal position $Y$ of ${x}$, respectively, as follows:
\begin{equation}\label{eq2.1}
{\rho}^2=\fr m{m-1}\left(|{h}|^2-m|{H}|^2\right),\quad {Y}={\rho}(1, x)\in \bbr^1_1\times\bbr^{m+2}_1\equiv\bbr^{m+3}_2.
\end{equation}
Then ${Y}(M^m)$ is clearly included in the light cone $\mathbb{C}^{m+2}\subset\bbr^{m+3}_2$ where
$$
\mathbb{C}^{m+2}=\{\xi\in\bbr^{m+3}_2;\lagl \xi,\xi\ragl_2=0,\xi\neq0\}.
$$
The positivity of $\rho$ implies that ${Y}:M^m\to \bbr^{m+3}_2$ is an immersion of $M^m$ into the $\bbr^{m+3}_2$.
Clearly, the metric ${g}:=\lagl d{Y},d{Y}\ragl_2\equiv{\rho}^2\lagl dx,dx\ragl_1$ on $M^m$, induced by $Y$ and called the conformal metric, is invariant under the pseudo-orthogonal group $O(m+3,2)$ of linear transformations on $\bbr^{m+3}_2$ reserving the Lorentzian product $\lagl\cdot,\cdot\ragl_2$. Such kind of things are called the conformal invariants of $ x$.

For any local orthonormal frame field $\{e_i\}$ and the dual $\{\theta^i\}$ on $M^m$ with respect to the standard metric $\lagl d x,d x\ragl_1$, define
\begin{equation}
E_i={\rho}^{-1}e_i,\quad \omega^i={\rho}\theta^i.
\end{equation}
Then $\{E_i\}$ is a local orthonormal frame field with respect to the conformal metric ${g}$ with $\{\omega^i\}$ its dual coframe.
Let ${n}$ be the time-like unit normal of ${x}$. Define
$$ {\xi}=\left(-{H},-{H} x+{n}\right),$$
then $\lagl {\xi},{\xi} \ragl_2=-1$.
Let $\Delta$ denote the Laplacian with respect to  the conformal metric ${g}$. Define
${N}:M^{m}\to\bbr^{m+3}_2$ by
\begin{equation}
{N}=-\fr1m\Delta {Y}-\fr1{2m^2}\lagl\Delta {Y},\Delta {Y}\ragl_2{Y}.
\end{equation}
Then it holds that
\begin{equation}
\lagl \Delta {Y},{Y}\ragl_2=-m,\quad
\lagl {Y},{Y}\ragl_2=\lagl {N},{N}\ragl_2=0,\quad \lagl {Y},{N}\ragl_2=1.
\end{equation}
Furthermore, $\{{Y},{N},{Y}_{i},{\xi},\  1\leq i\leq m\}$ forms a moving frame in $\bbr^{m+3}_2$ along ${Y}$, with respect to which the equations of motion are as follows:
\begin{equation}\label{eq2.5}
\left\{\begin{aligned}
d{Y}=&\,\sum{Y}_i\omega^i,\\
d{N}=&\,\sum\psi_i{Y}_i+\Phi  \xi,\\
d{Y}_i=&-\psi_i{Y}-\omega_i {N}+\sum\omega_{ij}{Y}_j+\tau_i  \xi,\\
d \xi=&\,\Phi {Y}+\sum\tau_i{Y}_i.
\end{aligned}\right.
\end{equation}
By the exterior differentiation of \eqref{eq2.5} and using Cartan's lemma, we can write
\begin{equation}
\begin{aligned}
&\Phi=\sum\Phi_i\omega^i,\quad \psi_i=\sum A_{ij}\omega^j,\quad A_{ij}=A_{ji};\\
&\tau_i=\sum B_{ij}\omega^j,\quad B_{ij}=B_{ji}.
\end{aligned}
\end{equation}
Then the conformal form $\Phi$, the Blaschke tensor $A$ and the conformal second fundamental form $B$ defined by
$$\Phi=\sum\Phi_i\omega^i,\quad A=\sum A_{ij}\omega^i\omega^j,\quad B=\sum B_{ij}\omega^i\omega^j$$
are clearly conformal invariants. By a long but direct computation, we find that
\begin{align}
A_{ij}=&-\lagl{Y}_{ij},{N}\ragl_2=-{\rho}^{-2}((\log {\rho})_{,ij}-e_i(\log {\rho})e_j(\log {\rho})+{h}_{ij}{H})\nnm\\
&-\fr12{\rho}^{-2}\left(|\bar{\nabla}\log {\rho}|^2 -|H|^2-1\right)\delta_{ij},\\
B_{ij}=&-\lagl{Y}_{ij},{\xi}\ragl_2={\rho}^{-1}({h}_{ij}-{H}\delta_{ij}),\label{eq2.8}\\
\Phi_i=&-\lagl {\xi},d{N}\ragl_2=-{\rho}^{-2}[({h}_{ij}-{H}\delta_{ij})e_j(\log {\rho})+e_i({H})],\label{eq2.9}
\end{align}
where $Y_{ij}=E_j(Y_i)$, $\bar\nabla$ is the Levi-Civita connection of the induced metric from $\lagl\cdot,\cdot\ragl_1$, and the subscript ${}_{,ij}$ denotes the covariant derivatives with respect to $\bar\nabla$. The differentiation of \eqref{eq2.5} also gives the following integrability conditions:
\begin{align}
&\Phi_{ij}-\Phi_{ji}
=\sum(B_{ik}A_{kj}-B_{kj}A_{ki}),\label{Phiij}\\
&A_{ijk}-A_{ikj}=B_{ij}\Phi_k-B_{ik}\Phi_j,\label{Aijk}\\
&B_{ijk}-B_{ikj}=\delta_{ij}\Phi_k-\delta_{ik}\Phi_j,\label{Bijk}\\
&R_{ijkl}=B_{ik}B_{jl}
-B_{il}B_{jk}+A_{il}\delta_{jk}-A_{ik}\delta_{jl}
+A_{jk}\delta_{il}-A_{jl}\delta_{ik},\label{Rijkl}
\end{align}
where $A_{ijk}$, $B_{ijk}$, $\Phi_{ij}$ are respectively the components of the covariant derivatives of $A$, $B$, $\Phi$, and $R_{ijkl}$ is the components of the Riemannian curvature tensor of the conformal metric $g$. Furthermore, by \eqref{eq2.1} and \eqref{eq2.8} we have
\be\label{trB}
\tr B=\sum B_{ii}=0,\quad |B|^2=\sum(B_{ij})^2=\fr{m-1}{m},
\ee
and by \eqref{Rijkl} we find the Ricci curvature tensor
\be\label{Rij}
R_{ij}=\sum B_{ik}B_{kj}+\delta_{ij}\tr
A+(m-2)A_{ij},
\ee
which implies that
\be\label{trA}
\tr A=\fr{1}{2m}(m^2\kappa-1)
\ee
with $\kappa$ being the normalized scalar curvature of $g$.

It is easily seen (\cite{ag1}) that the conformal position vector ${Y}$ defined above is exactly the canonical lift of the composition map of
$\bar x=\sigma_1\circ{x}:M^m\to\mathbb{Q}^{m+1}_1$, implying that the conformal invariants ${g},\Phi,A,B$ defined above are the same as those of $\bar x$ introduced by Nie at al in \cite{ag1}.

{\dfn[cf. \cite{ag1}] Let $x,\td x:M^m\to \mathbb{S}^{m+1}_1$ be two regular space-like hypersurfaces with $Y,\td Y$ their conformal position vectors, respectively. If there exists some element $\mathbb{T}\in O(m+3,2)$ such that $\td Y=\mathbb{T}(Y)$, then $x,\td x$ are called conformal equivalent to each other.}

As it has appeared in \cite{LS}, the conformal space $\mathbb{Q}^{m+1}_1$ is apparently covered by the following two non-homogeneous coordinate systems $(U_{\alpha},\strl{(\alpha)}{\Psi})$ modeled on the de Sitter space $\mathbb{S}^{m+1}_1$, $\alpha=1,2$, where
\begin{equation}
\begin{aligned}
U_1=\left\{[y]\in\mathbb{Q}^{m+1}_1;y=(y_1,y_2,y_3)\in \bbr^1_1\times\bbr^1_1\times\bbr^{m+1}\equiv\bbr^{m+3}_2,y_1\neq0\right\},\\
U_2=\left\{[y]\in\mathbb{Q}^{m+1}_1;y=(y_1,y_2,y_3)\in \bbr^1_1\times\bbr^1_1\times\bbr^{m+1}\equiv\bbr^{m+3}_2,y_2\neq0\right\},
\end{aligned}
\end{equation}
and the diffeomorphisms
$\strl{(\alpha)}{\Psi}$: $U_{\alpha}\to \mathbb{S}^{m+1}_1$, $\alpha=1,2$,
are defined by
\begin{align}
\strl{{ (1)}}{\Psi}([y])=y^{-1}_1(y_2,y_3), \text{ for } [y]\in U_1,\ y=(y_1,y_2,y_3);\\
\strl{{ (2)}}{\Psi}([y])=y^{-1}_2(y_1,y_3), \text{ for } [y]\in U_2,\ y=(y_1,y_2,y_3). \end{align}
Then, with respect to the conformal structure on $\mathbb{Q}^{m+1}_1$ introduced in \cite{ag1} and the standard metric on $\mathbb{S}^{m+1}_1$, both $\strl{(1)}{\Psi}$ and $\strl{(2)}{\Psi}$ are conformal.

Now for a regular space-like hypersurface $\bar x:M^m \to\mathbb{Q}^{m+1}_1$ with the canonical lift $$Y:M^m\to\mathbb{C}^{m+2}\subset\bbr^{m+3}_2,$$
write $Y=(Y_1,Y_2,Y_3)\in \bbr^1_1\times\bbr^1_1\times\bbr^{m+1}$. Then we have the following two composed hypersurfaces:
\begin{equation}
\strl{{ (\alpha)}}{x}:=\strl{{ (\alpha)}}{\Psi}\! \circ \,\,\bar x|_{\strl{{ (\alpha)}}{M}}:\strl{{ (\alpha)}}{M}\to \mathbb{S}^{m+1}_1, \quad
\strl{{ (\alpha)}}{M}=\{p\in M^{m}; \bar x(p)\in U_{\alpha}\},\quad\alpha=1,2.
\end{equation}
Then $M^m=\strl{(1)}{M}\bigcup \strl{(2)}{M}$, and the following lemma is clearly true by a direct computation:

\begin{lem}[\cite{LS}] The conformal position vector $\strl{{ (1)}}{Y}$ of $\strl{{ (1)}}{x}$ is nothing but $Y|_{\strl{{ (1)}}{M}}$, while the conformal position vector $\strl{{ (2)}}{Y}$ of $\strl{{ (2)}}{x}$ is given by
\begin{equation}
\strl{{ (2)}}{Y}=\mathbb{T}(Y|_{\strl{{(2)}}{M}}),\mbox{ where } \mathbb{T}=\left( \begin{tabular}{c:c}$\begin{matrix} 0&1\\1&0\end{matrix}$&$0$\\
\hdashline
$0$&$I_{m+1}$\end{tabular}\right).
\end{equation}
\end{lem}

\begin{cor}[\cite{LS}]\label{cor2.2}
The basic conformal invariants $g,\Phi,A,B$ of $\bar x$ coincide accordingly with those of each of $\strl{{ (1)}}{x}$ and $\strl{{ (2)}}{x}$ on where $\strl{{ (1)}}{x}$ or $\strl{{ (2)}}{x}$ is defined, respectively.
\end{cor}

Therefore, $\strl{{ (1)}}{\Psi}$ and $\strl{{ (2)}}{\Psi}$ can be viewed as two non-homogenous coordinate maps preserving the conformal invariants of the regular space-like hypersurfaces. Thus we have

\begin{cor}
$\strl{{ (1)}}{x}$ and $\strl{{ (2)}}{x}$ are conformal equivalent to each other on $\strl{{ (1)}}{M}\cap \strl{{ (2)}}{M}$.
\end{cor}

On the other hand, up to conformal equivalences, all the regular space-like hypersurfaces in the three Lorentzian space forms can be viewed as ones in $\mathbb{Q}^{m+1}_1$ via the conformal embeddings $\sigma_1$, $\sigma_0$ and $\sigma_{-1}$ defined in \eqref{eq1.1}. Now, using $\strl{{ (1)}}{\Psi}$ and $\strl{{ (2)}}{\Psi}$, one can shift the conformal geometry of regular space-like hypersurfaces in $\mathbb{Q}^{m+1}_1$ to that of regular space-like hypersurfaces in the de Sitter space $\mathbb{S}^{m+1}_1$. It follows that, in a sense, the conformal geometry of regular space-like hypersurfaces can also be unified as that of the corresponding hypersurfaces in the de Sitter space.
Concisely, we can achieve this simply by introducing the following four conformal maps:
\begin{align}
&\strl{{ (1)}}{\sigma}=\strl{{ (1)}}{\Psi}\!\circ \,\,\sigma_0:\ \strl{{ (1)}}{\bbr}\!\!{}^{m+1}_1\rightarrow \mathbb{S}^{m+1}_1 ,\quad
u\mapsto\left(\fr{2u}{1+\lagl u,u\ragl},\fr{1-\lagl u,u\ragl}{1+\lagl u,u\ragl}\right),
\\
&\strl{{ (2)}}{\sigma}=\strl{{ (2)}}{\Psi}\!\circ \,\,\sigma_0:\ \strl{{ (2)}}{\bbr}\!\!{}^{m+1}_1\rightarrow \mathbb{S}^{m+1}_1 ,\quad
u\mapsto\left(\fr{1+\lagl u,u\ragl}{2u_1},\fr{u_2}{u_1},\fr{1-\lagl u,u\ragl}{2u_1}\right),
\\
&\strl{{ (1)}}{\tau}=\strl{{ (1)}}{\Psi}\!\circ \,\,\sigma_{-1}:\ \strl{{ (1)}}{\mathbb{H}}\!\!{}^{m+1}_1\rightarrow \mathbb{S}^{m+1}_1 ,\quad
y\mapsto\left(\fr{y_2}{y_1},\fr{y_3}{y_1},\fr{1}{y_1}\right),
\\
&\strl{{ (2)}}{\tau}=\strl{{ (2)}}{\Psi}\!\circ \,\,\sigma_{-1}:\ \strl{{ (2)}}{\mathbb{H}}\!\!{}^{m+1}_1\rightarrow \mathbb{S}^{m+1}_1 ,\quad
y\mapsto\left(\fr{y_1}{y_2},\fr{y_3}{y_2},\fr{1}{y_2}\right),
\end{align}
where
\begin{align}
&\strl{{ (1)}}{\bbr}\!\!{}^{m+1}_1=\{u\in \bbr^{m+1}_1;\ 1+\lagl u,u \ragl\neq0\},\quad
\strl{{ (2)}}{\bbr}\!\!{}^{m+1}_1=\{u=(u_1,u_2)\in \bbr^{m+1}_1;\ u_1\neq0\},\\
&\strl{{ (1)}}{\mathbb{H}}\!\!{}^{m+1}_1=\{y=(y_1,y_2,y_3)\in \mathbb{H}^{m+1}_1;\ y_1\neq0\},\quad
\strl{{ (2)}}{\mathbb{H}}\!\!{}^{m+1}_1=\{y=(y_1,y_2,y_3)\in \mathbb{H}^{m+1}_1;\ y_2\neq0\}.
\end{align}

The following theorem will be used later in this paper:

\begin{thm}[\cite{ag2}]\label{thm2.4} Two hypersurfaces $x:M^m \rightarrow \mathbb{S}^{m+1}_1$ and $\td x:\td{M}^m \rightarrow \mathbb{S}^{m+1}_1$ $(m\geq3)$ are conformal equivalent if and only if there exists a diffeomorphism $f:M\rightarrow \td{M}$ which preserves the conformal metric and the conformal second fundamental form.
\end{thm}

For a regular space-like hypersurface $x:M^m \rightarrow \mathbb{S}^{m+1}_1$, one calls $D^\lambda:=A+\lambda B$ a {\it para-Blaschke tensor} of $x$ with a real parameter $\lambda$ (cf. \cite{ZS}). From \eqref{Aijk} and \eqref{Bijk}, we have then
\begin{equation}
\label{Dijk}
D^\lambda_{ijk}-D^\lambda_{ikj}=
(B_{ij}+\lambda\delta_{ij})\Phi_k-(B_{ik}+\lambda\delta_{ik})\Phi_j,
\end{equation}
where $D^\lambda_{ijk}$ are components of the covariant derivatives of $D^\lambda$.

\section{Examples}

In this section, we shall list three kinds of regular space-like hypersurfaces in $\mathbb{S}^{m+1}_1$ two of which are new up to now. Necessary computations are presented in detail to find their conformal invariants. In particular, it is shown that these hypersurfaces are all of parallel para-Blaschke tensors and, in particular, of vanishing conformal forms.

{\expl[\cite{ag1}, cf.\cite{hl04}]\label{expl3.1}\rm Let $\bbr^+$ be the half line of
positive real numbers. For any two given natural numbers $p,q$
with $p+q<m$ and a real number $a>1$, consider the hypersurface of warped product embedding
$$\dps u:\mathbb{H}^q\left(-\fr1{a^2-1}\right)\times\mathbb{S}^p(a)\times \bbr^+\times\bbr^{m-p-q-1}\to
\bbr^{m+1}_1$$
defined by
$$u=(tu',tu'',u'''),\mb{ where }
u'\in \mathbb{H}^q\left(-\fr1{a^2-1}\right),\ u''\in \mathbb{S}^p(a),\ t\in\bbr^+,\ u'''\in\bbr^{m-p-q-1}.$$
Then $\bar x:=\sigma_0\circ u$ is a regular space-like hypersurface in the conformal space $\mathbb{Q}^{m+1}_1$ with parallel conformal second fundamental form. This hypersurface is denoted as $WP(p,q,a)$ in \cite{ag1}. By Proposition 3.1 in \cite{ag1} together with the argument in its proof, $\bar x$ is also of parallel Blaschke tensor. It follows from Corollary \ref{cor2.2} that the composition map $$x=\Psi\circ\bar x:\mathbb{H}^q\left(-\fr1{a^2-1}\right)\times\mathbb{S}^p(a)\times \bbr^+\times\bbr^{m-p-q-1}\to \mathbb{S}^{m+1}_1,$$
where $\Psi$ denotes $\strl{(1)}{\Psi}$ or $\strl{(2)}{\Psi}$, defines a regular space-like hypersurface in $\mathbb{S}^{m+1}_1$ with both parallel conformal second fundamental form and parallel Blaschke tensor, implying the identically vanishing of the conformal form. Then it follows that the para-Blaschke tensor $D^\lambda$ of $x$ for any $\lambda$ is parallel. Note that, by a direct calculation, one easily finds that both $WP(p,q,a)$ and $x$ has exactly three distinct conformal principal curvatures.

{\expl\label{expl3.2}\rm Given $r>0$. For any integers $m$ and $K$ satisfying $m\geq
3$ and $2\leq K\leq m-1$, let $\td y_1:M^{K}_1\to
\mathbb{S}^{K+1}_1(r)\subset\bbr^{K+2}_1$ be a regular space-like hypersurface with constant scalar curvature $\td S_1$ and the  mean curvature $\td H_1$ satisfy}
\begin{equation}\label{ex1S}
\td S_1=\fr{mK(K-1)+(m-1)r^2}{mr^2}-m(m-1)\lambda^2,\quad \td H_1=\fr{m}{K}\lambda
\end{equation}
and
$$
\td y=(\td y_0,\td y_2):
\mathbb{H}^{m-K}\left(-\fr1{r^2}\right)\to\bbr^1_1\times\bbr^{m-K}\equiv\bbr^{m-K+1}_1
$$
be the canonical embedding, where $\td y_0>0$. Set
\begin{equation}\label{eq3.2}
\td M^m=M^{K}_1\times \mathbb{H}^{m-K}\left(-\fr1{r^2}\right),\quad\td Y=(\td
y_0,\td y_1,\td y_2).
\end{equation}
Then $\td Y:\td M^m\to\bbr^{m+3}_2$ is an immersion satisfying
$\lagl \td Y,\td Y\ragl_2=0$. The induced metric
$$g=\lagl d\td Y,d\td Y\ragl_2=-d\td y_0^2+\lagl d\td y_1,d\td y_1\ragl_1+d\td y_2\cdot d\td y_2$$
by $\td Y$ is clearly a Riemannian one, and thus as Riemannian manifolds we have
\begin{equation}
\label{exm1}
(\td M^m,g)=(M_1,\lagl d\td y_1,d\td y_1\ragl_1)\times \left(\mathbb{H}^{m-K}\left(-\fr1{r^2}\right),
\lagl d\td y,d\td y\ragl_1\right).
\end{equation}
Define
\begin{equation}
\td x_1=\fr{\td y_1}{\td y_0},\quad \td
x_2=\fr{\td y_2}{\td y_0}, \quad \td x=(\td x_1,\td x_2).
\end{equation}
Then we have a smooth map $\td x:M^m\to \mathbb{S}^{m+1}_1$. Furthermore,
\begin{equation}
\label{dx}
d\td x=-\fr{d\td y_0}{\td y^2_0}(\td y_1,\td y_2) +\fr1{\td
y_0}(d\td y_1,d\td y_2).
\end{equation}
Therefore the induced ``metric'' $\td g=\lagl d\td x, d\td x\ragl_1$ is derived as
\begin{equation}
\begin{aligned}
\td g=\ &\td y^{-4}_0d\td y^2_0(\lagl \td y_1,\td y_1\ragl_1+\td y_2\cdot \td y_2)+\td y^{-2}_0(\lagl d\td y_1,d\td y_1\ragl_1+d\td y_2\cdot d\td y_2)\\
&-2\td y^{-3}_0d\td y_0(\lagl \td y_1,d\td y_1\ragl_1+\td y_2\cdot d\td y_2)\\
=\ &\td y^{-2}_0(d\td y^2_0+g+d\td y^2_0-2d\td y^2_0)\\
=\ &\td y^{-2}_0g,
\end{aligned}
\end{equation}
implying that $\td x$ is a regular space-like hypersurface.

If $\td n_1$ is the time-like unit normal vector field of $\td y_1$ in
$\mathbb{S}^{K+1}_1(r)\subset\bbr^{K+2}_1$, then $\td n=(\td
n_1,0)\in\bbr^{m+2}_1$ is a time-like unit normal vector field of $\td x$.
Consequently, by \eqref{dx}, the second fundamental form $\td h$ of $\td
x$ is given by
\begin{equation}
\begin{aligned}
\td h=\lagl d\td n, d\td x\ragl_1&=\lagl(d\td n_1,0),-\td y^{-2}_0d\td y_0(\td y_1,\td y_2)+\td y^{-1}_0(d\td y_1,d\td y_2)\ragl_1\\
&=-\td y^{-2}_0d\td y_0\lagl d\td n_1,\td y_1\ragl_1+\td y^{-1}_0\lagl d\td n_1,d\td y_1\ragl_1\\
& =\td y^{-1}_0h,
\end{aligned}
\end{equation}
where $h$ is the second fundamental form of $\td y_1:M^{K}_1\rightarrow\mathbb{S}^{K+1}_1$.

Let $\{E_i\;;\ 1\leq i\leq K\}$ (resp. $\{E_i\;;\ K+1\leq i\leq
m\}$) be a local orthonormal frame field on $(M_1,d\td y^2_1)$
(resp. on $\mathbb{H}^{m-K}(-\fr1{r^2})$). Then $\{E_i\;;\ 1\leq i\leq m\}$
gives a local orthonormal frame field on $(\td M^m,g)$. Put $e_i=\td
y_0E_i$, $i=1,\cdots,m$. Then $\{e_i\;;\ 1\leq i\leq m\}$ is a local orthonormal frame field along $\td x$.
 Thus we obtain
\begin{equation}\label{hij1}
\td h_{ij}=\td h(e_i,e_j)=\td y^2_0\td h(E_i,E_j)=\begin{cases}\td y_0
h(E_i,E_j)=\td y_0h_{ij},\quad &\mb{when\ } 1\leq i,j\leq K,\\0,\quad &\mb{otherwise\ }.
\end{cases}
\end{equation}
The mean curvature $\td H$ of $\td x$ is given by
\begin{equation}\label{tdH1}
\td H=\fr Km\td y_0\td H_1=\td y_0\lambda.
\end{equation}
Therefore, by definition, the conformal factor $\td \rho$ of $\td x$ is determined by
\begin{equation}
\td \rho^2=\fr{m}{m-1}\left(\sum_{i,j=1}^m\td h^2_{ij}-m|\td
H|^2\right)=\fr{m}{m-1}\td y^2_0(\sum_{i,j=1}^{K}h_{ij}^2-m\lambda^2) =\td
y^2_0,
\end{equation}
where we have used the Gauss equation and \eqref{ex1S}.
It follows that $\td x$ is regular and its conformal factor is
\begin{equation}\label{ex1po}
\td\rho=\td y_0.
\end{equation}
Thus $\td Y$, given in \eqref{eq3.2}, is exactly the conformal position vector of $\td x$, implying the induced metric $g$ by $\td Y$ is nothing but
the conformal metric of $\td x$. Furthermore, the conformal second fundamental form of
$\td x$ is given by
\begin{equation}
\label{ex1B}
 B=\td \rho^{-1}\sum_{i,j=1}^m(\td h_{ij}-\td H\delta_{ij})\omega^i\omega^j=\sum_{i,j=1}^K(h_{ij}-\lambda\delta_{ij})\omega^i\omega^j -\sum_{i=K+1}^m\lambda(\omega^i)^2,
\end{equation}
where $\{\omega^i\}$ is the local coframe field on $M^m$ dual to $\{E_i\}$.

On the other hand, by \eqref{exm1} and the Gauss equations of $\td y_1$ and
$\td y$, one finds that the Ricci tensor of $g$ is given as follows:
\begin{align}
R_{ij}=&\fr{K-1}{r^2}\delta_{ij}-m\lambda h_{ij}+\sum_{k=1}^{K}h_{ik}h_{kj},\mb{ if }
1\leq i,j\leq K,\\
R_{ij}=&-\fr{m-K-1}{r^2}\delta_{ij},\mb{ if } K+1\leq
i,j\leq m,\\
R_{ij}=&0,\ \mb{if $1\leq i\leq K$, $K+1\leq j\leq m$, or
$K+1\leq i\leq m$, $1\leq j\leq K$}
\end{align}
which implies that the normalized scalar curvature of $g$ is given by
\begin{equation}
\kappa=\fr{m(K(K-1)-(m-K)(m-K-1))+(m-1)r^2}{m^2(m-1)r^2}-\lambda^2.
\end{equation}
Thus
\begin{equation}
\label{ex1k}
\fr1{2m}(m^2\kappa-1)=\fr{K(K-1)-(m-K)(m-K-1)}{2(m-1)r^2}-\fr12m\lambda^2.
\end{equation}

Since $m\geq 3$, it follows from \eqref{Rij} and \eqref{ex1B}--\eqref{ex1k}
that the Blaschke tensor of $\td x$ is given by $ A=\sum A_{ij}\omega^i\omega^j$, where
\begin{align}
A_{ij}=&\left(\fr{1}{2r^2}+\fr12\lambda^2\right)\delta_{ij}-\lambda h_{ij},\quad\mb{if\ } 1\leq i,j\leq K;\\
A_{ij}=&\left(-\fr{1}{2r^2}+\fr12\lambda^2\right)\delta_{ij},\quad\mb{if\ } K+1\leq i,j\leq m; \\
A_{ij}=&0,\quad\mb{if $1\leq i\leq K$, $K+1\leq j\leq m$, or
$K+1\leq i\leq m$, $1\leq j\leq K$}.
\end{align}
Therefore,  the para-Blaschke tensor
$D^\lambda=A+\lambda B=\sum D^\lambda_{ij}\omega^i\omega^j$ satisfies
\begin{equation}
\begin{aligned}
 D^\lambda_{ij}=&A_{ij}+\lambda B_{ij}=\left(\fr{1}{2r^2}-\fr12\lambda^2\right)\delta_{ij},\quad\mb{for\
} 1\leq i,j\leq K;\\
 D^\lambda_{ij}=& A_{ij}+\lambda
 B_{ij}=-\left(\fr{1}{2r^2}+\fr12\lambda^2\right)\delta_{ij},\quad\mb{for\
} K+1\leq
i,j\leq m; \\
 D^\lambda_{ij}=&0,\quad\mb{for $1\leq i\leq K$, $K+1\leq j\leq m$,
or
$K+1\leq i\leq m$, $1\leq j\leq K$}.
\end{aligned}
\end{equation}
Thus, we know that $ D^\lambda$ has exactly two different para-Blaschke eigenvalues which are constant. Then it follows easily that the para-Blaschke tensor $ D^\lambda$ of $\td x$ is parallel.

Now, by the way, we would like to make a direct computation of the conformal form $\Phi=\sum\Phi_i\omega^i$ of $\td x$. From \eqref{hij1},\eqref{tdH1} and \eqref{ex1po}, we have
$$
\td h_{ij}-\td H\delta_{ij}=\begin{cases}\td\rho (h_{ij}-\lambda\delta_{ij}),\quad &\mb{for\ }1\leq i,j\leq K,\\-\td\rho\lambda\delta_{ij},\quad &\mb{for\ }K+1\leq i,j\leq m,\\0,\quad &\mb{otherwise\ }.
\end{cases}
$$

(1) $1\leq i\leq K$. we compute using the formula \eqref{eq2.9}:
\begin{equation}
\begin{aligned}
\Phi_i=&-\td\rho^{-2}\left(\sum_{j=1}^m(\td{h}_{ij}-\td{H}\delta_{ij})e_j(\log {\td \rho})+e_i(\td{H})\right)\\
=&-\td\rho^{-2}\left(\td\rho\sum_{j=1}^K({h}_{ij}-\lambda\delta_{ij})e_j(\log {\td\rho})+\lambda e_i(\td\rho)\right)\\
=&-\td\rho^{-2}\left(\sum_{j=1}^K({h}_{ij}-\lambda\delta_{ij})e_j(\td\rho)+\lambda e_i(\td\rho)\right)\\
=&-\td\rho^{-2}\sum_{j=1}^K{h}_{ij}e_j(\td y_0)=0.
\end{aligned}
\end{equation}

(2) $K+1\leq i\leq m$. We have
\begin{equation}
\begin{aligned}
\Phi_i=&-\td\rho^{-2}\left(\sum_{j=1}^m(\td{h}_{ij}-\td{H}\delta_{ij})e_j(\log \td\rho)+e_i(\td{H})\right)\\
=&-\td\rho^{-2}\left(-\td\rho\lambda \sum_{j=K+1}^m\delta_{ij}e_j(\log {\td\rho})+\lambda e_i(\td\rho)\right)
=0.
\end{aligned}
\end{equation}

Thus the conformal form $\Phi$ of $\td x$ vanishes identically.

{\expl\label{expl3.3}\rm Given $r>0$. For any integers $m$ and $K$ satisfying $m\geq
3$ and $2\leq K\leq m-1$, let
$$\td y:M^{K}_1\to
\mathbb{H}^{K+1}_1\left(-\fr1{r^2}\right)\subset\bbr^{K+2}_2$$
be a regular space-like hypersurface with constant scalar curvature $\td S_1$ and the  mean curvature $\td H_1$ satisfy}
\begin{equation}
\label{ex2S}
\td S_1=\fr{-mK(K-1)+(m-1)r^2}{mr^2}-m(m-1)\lambda^2,\quad \td H_1=\fr{m}{K}\lambda
\end{equation}
and
$$
\td y_2: \mathbb{S}^{m-K}(r)\to\bbr^{m-K+1}
$$
be the canonical embedding. Set
\begin{equation}
\td M^m=M^{K}_1\times \mathbb{S}^{m-K}(r),\quad\td Y=(\td y,\td y_2).
\end{equation}
Then $\lagl \td Y,\td Y\ragl_2=0$. Thus we have an immersion $\td Y:M^m\to \mathbb{C}^{m+2}\subset\bbr^{m+3}_2$ with the induced metric
$$g=\lagl d\td Y,d\td Y\ragl_2=\lagl d\td y,d\td y\ragl_2+d\td y_2\cdot d\td y_2,$$
which is certainly positive definite. It follows that, as Riemannian manifolds
\begin{equation}
\label{exm2}
(\td M^m,g)=(M_1,\lagl d\td y,d\td y\ragl_2)\times
\left(\mathbb{S}^{m-K}(r),d\td y^2_2\right).
\end{equation}

If we write $\td y=(\td y_0,\td y'_1,\td y''_1)\in\bbr^1_1\times\bbr^1_1\times\bbr^K\equiv\bbr^{K+2}_2$, then $\td y_0$ and $\td y'_1$ can not be zero simultaneously. So, without loss of generality, we can assume that $\td y_0\neq0$. In this case, we denote $\veps=\sgn(\td y_0)$ and write $\td y_1:=(\td y'_1,\td y''_1)$. Define
\begin{equation}
\td x_1=\fr{\td y_1}{\td y_0},\quad \td
x_2=\fr{\td y_2}{\td y_0}, \quad \td x=\varepsilon(\td x_1,\td x_2).
\end{equation}
Then similar to that in Example \ref{expl3.2}, $\td x:\td M^m\to \mathbb{S}^{m+1}_1$ defines a regular space-like hypersurface. In fact, since
\begin{equation}
\varepsilon d\td x=-\fr{d\td y_0}{\td y^2_0}(\td y_1,\td y_2) +\fr1{\td
y_0}(d\td y_1,d\td y_2),
\end{equation}
the induced metric $\td g=\lagl d\td x, d\td x\ragl_1$ is related to $g$ by
\begin{align}
\td g=&\td y^{-4}_0d\td y^2_0(\lagl \td y_1,\td y_1\ragl_1+\td y_2\cdot\td y_2)+\td
y^{-2}_0(\lagl d\td y_1,d\td y_1\ragl_1+d\td y_2\cdot d\td y_2)\nnm\\
&-2\td y^{-3}_0d\td y_0(\lagl \td y_1,d\td y_1\ragl_1+ \td y_2\cdot d\td y_2)\nnm\\
=&\td y^{-2}_0(-d\td y^2_0+\lagl d\td y_1,d\td y_1\ragl_1+d\td y_2\cdot d\td y_2)\nnm\\
=&\td y^{-2}_0g.
\end{align}

Suitably choose the time-like unit normal vector field  $(\td n_0,\td n_1)$ of $\td y$, define
$$\td n=(\td n_1,0)-\varepsilon\td
n_0\td x\in\bbr^{m+2}_1.$$
Then $\lagl \td n,\td n\ragl_1=-1,\lagl \td n,\td x\ragl_1=0,\lagl \td n,d\td x\ragl_1=0$ indicating that $\td n$ is a time-like unit normal vector field of $\td x$.
The second fundamental form $\td h$ of $\td x$ is given by
\begin{align}
\td h=&\lagl d\td n, d\td x\ragl_1=\lagl (d\td n_1,0)-\varepsilon d\td n_0\td x-\varepsilon\td n_0d\td x,d\td x\ragl_1\nnm\\
=&\lagl(d\td n_1,0),d\td x\ragl_1-\varepsilon d\td n_0\lagl \td x,d\td x\ragl_1-\varepsilon\td n_0\lagl d\td x,d\td x\ragl_1\nnm\\
=&\varepsilon(\lagl d\td n_1, -\td y_0^{-2}d\td y_0\td y_1+\td y_0^{-1}d\td y_1\ragl_1-\td n_0\lagl d\td x,d\td x\ragl_1)\nnm\\
=&\varepsilon(-\td y_0^{-2}d\td y_0\lagl d\td n_1,\td y_1\ragl_1+\td y_0^{-1}\lagl d\td n_1, d\td y_1\ragl_1-\td n_0\lagl d\td x,d\td x\ragl_1)\nnm\\
=&\varepsilon(-\td y_0^{-2}d\td y_0(d\td n_0\cdot\td y_0)+\td y_0^{-1}\lagl d\td n_1, d\td y_1\ragl_1-\td n_0\lagl d\td x,d\td x\ragl_1)\nnm\\
=&\varepsilon(\td y^{-1}_0(-d\td n_0\cdot d\td y_0+\lagl d\td n_1,d\td y_1\ragl_1)-\td n_0\lagl d\td x,d\td x\ragl_1)\nnm\\
=&\varepsilon(\td y^{-1}_0\lagl d(\td n_0,\td n_1),d\td y\ragl_1-\td n_0\lagl d\td x,d\td x\ragl_1)\nnm\\
=&\varepsilon(\td y^{-1}_0h-\td n_0\td y^{-2}_0g)
\end{align}
where we have used $-d\td n_0\cdot \td y_0+\lagl d\td n_1,\td y_1\ragl_1=0$ and $h$ is the second fundamental form of $\td y$.

Let $\{E_i\;;\ 1\leq i\leq K\}$ (resp. $\{E_i\;;\ K+1\leq i\leq
m\}$) be a local orthonormal frame field on $(M_1,d\td y^2)$ (resp.
on $\mathbb{S}^{m-K}(r)$). Then $\{E_i\;;\ 1\leq i\leq m\}$ is a local
orthonormal frame field on $(M^m,g)$. Put $e_i=\varepsilon\td y_0E_i$,
$i=1,\cdots,m$. Then $\{e_i\;;\ 1\leq i\leq m\}$ is a local
orthonormal frame field with respect to the metric $\td g=\lagl d\td x,d\td x\ragl_1$. Thus
\begin{equation}\label{hij2}
\td h_{ij}=\td h(e_i,e_j)=\td y^2_0\td h(E_i,E_j)=\begin{cases}\varepsilon(\td y_0h_{ij}-\td n_0\delta_{ij}),\quad &\mb{when\ } 1\leq i,j\leq K,\\
-\varepsilon\td n_0g(E_i,E_j)=-\varepsilon\td n_0\delta_{ij},\quad &\mb{when\ }K+1\leq i,j\leq m,\\
0,\quad &\mb{for other}\quad i,j.
\end{cases}
\end{equation}
The mean curvature $\td H$ of $\td x$ is
\begin{equation}\label{tdH2}
\td H=\fr1m\sum_{i=1}^m\td h_{ii}=\varepsilon\fr{1}{m}(\td y_0K\td H_1-K\td n_0)-\varepsilon\fr{1}{m}(m-K)\td n_0=\varepsilon(\td y_0\lambda-\td n_0)
\end{equation}
and
\begin{equation}
|\td h|^2=\sum^{K}_{i,j=1}\td y^2_0h^2_{ij}+\td n^2_0\delta^2_{ij}-2\td n_0\td y_0h_{ij}\delta_{ij}+\sum^m_{i,j=K+1}(-\td n_0)^2\delta^2_{ij}=\td y_0^2h^2-2m\lambda\td n_0\td y_0+m\td n^2_0.
\end{equation}
Therefore, by definition, the conformal factor $\td \rho$ of $\td x$ is
determined by
\begin{equation}
\td \rho^2=\fr{m}{m-1}\left(\sum_{i,j=1}^m\td h^2_{ij}-m|\td
H|^2\right)=\fr{m}{m-1}\td y^2_0(\sum_{i,j=1}^K h_{ij}^2-m\lambda^2) =\td y^2_0,
\end{equation}
where we have used the Gauss equation and \eqref{ex2S}.
Hence
\begin{equation}\label{ex2po}
\td \rho=|\td y_0|=\varepsilon\td y_0>0
\end{equation}
 and thus $\td Y=\td \rho(1,\td x)$ is the conformal position vector of $\td x$. Consequently, the conformal metric of $\td x$ is defined by $\lagl d\td Y,d\td Y\ragl_2=g$. Furthermore, the conformal second fundamental form of $\td x$ is given by
\begin{equation}
\label{ex2B}
B=\td \rho^{-1}\sum_{i,j=1}^m(\td h_{ij}-\td H\delta_{ij})\omega^i\omega^j=\sum_{i,j=1}^K(h_{ij}-\lambda\delta_{ij})\omega^i\omega^j -\sum_{i=K+1}^m\lambda(\omega^i)^2,
\end{equation}
where $\{\omega^i\}$ is the local coframe field on $M^m$ dual to $\{E_i\}$.

On the other hand, by \eqref{exm2} and the Gauss equations of $\td y_1$
and $\td y$, one finds the Ricci tensor of $g$ as
follows:
\begin{align}
R_{ij}=&-\fr{K-1}{r^2}\delta_{ij}-m\lambda h_{ij}+\sum_{k=1}^{K}h_{ik}h_{kj},\quad\mb{if\
} 1\leq i,j\leq K,\\
R_{ij}=&\fr{m-K-1}{r^2}\delta_{ij},\quad\mb{if\ } K+1\leq i,j\leq m,\\
R_{ij}=&0,\quad\mb{if $1\leq i\leq K$, $K+1\leq j\leq m$, or
$K+1\leq i\leq m$, $1\leq j\leq K$},
\end{align}
which implies that the normalized scalar curvature of $g$ is given by
\begin{equation}
\kappa=\fr{m((m-K)(m-K-1)-K(K-1))+(m-1)r^2}{m^2(m-1)r^2}-\lambda^2.
\end{equation}
Thus
\begin{equation}
\label{ex2k}
\fr1{2m}(m^2\kappa-1)=\fr{(m-K)(m-K-1)-K(K-1)}{2(m-1)r^2}-\fr12m\lambda^2.
\end{equation}

Since $m\geq 3$, it follows from \eqref{Rij} and \eqref{ex2B}--\eqref{ex2k}
that the Blaschke tensor of $\td x$ is given by $A=\sum A_{ij}\omega^i\omega^j$, where
\begin{align}
 A_{ij}=&\left(-\fr{1}{2r^2}+\fr12\lambda^2\right)\delta_{ij}-\lambda h_{ij},\quad\mb{if\ } 1\leq i,j\leq K;\\
A_{ij}=&\left(\fr{1}{2r^2}+\fr12\lambda^2\right)\delta_{ij},\quad\mb{if\ } K+1\leq i,j\leq m; \\
 A_{ij}=&0,\quad\mb{if $1\leq i\leq K$, $K+1\leq j\leq m$, or
$K+1\leq i\leq m$, $1\leq j\leq K$}.
\end{align}
Therefore,  the para-Blaschke tensor
$ D^\lambda= A+\lambda B=\sum  D^\lambda_{ij}\omega^i\omega^j$ satisfies
\begin{equation}
\begin{aligned}
 D^\lambda_{ij}=& A_{ij}+\lambda  B_{ij}=-\left(\fr{1}{2r^2}+\fr12\lambda^2\right)\delta_{ij},\quad\mb{for\
} 1\leq i,j\leq K;\\
 D^\lambda_{ij}=& A_{ij}+\lambda
 B_{ij}=\left(\fr{1}{2r^2}-\fr12\lambda^2\right)\delta_{ij},\quad\mb{for\
} K+1\leq
i,j\leq m; \\
 D^\lambda_{ij}=&0,\quad\mb{for $1\leq i\leq K$, $K+1\leq j\leq m$,
or
$K+1\leq i\leq m$, $1\leq j\leq K$}.
\end{aligned}
\end{equation}
It follows that the para-Blaschke tensor $ D^\lambda$ of $\td x$ is parallel.

Finally, we are to prove that the conformal form $\Phi\equiv 0$ by a direct computation, which is in some sense different from that in the last example.  To this end we first use \eqref{hij2}, \eqref{tdH2} and \eqref{ex2po} to find
$$
\td h_{ij}-\td H\delta_{ij}=\begin{cases}\td\rho (h_{ij}-\lambda\delta_{ij}),\quad &\mb{for\ }1\leq i,j\leq K,\\-\td\rho\lambda\delta_{ij},\quad &\mb{for\ }K+1\leq i,j\leq m,\\0,\quad &\mb{otherwise\ }.
\end{cases}
$$

Next, for any $p=(p_1,p_2)\in M_1\times M_2\equiv M^m$, we can suitably choose the local frame field $\{E_i,1\leq  i\leq K\}$ around $p_1\in M_1$ such that $h_{ij}(p)=h_i\delta_{ij}$, $1\leq i,j\leq K$. Then, for $1\leq i\leq K$, we have $E_i(\td n)=h_iE_i(\td y)$ at $p_1$ and, in particular, $E_i(\td n_0)=h_iE_i(\td y_0)$. We compute below at the arbitrarily given point $p$:

(1) $1\leq i\leq K$.
\begin{equation}
\begin{aligned}
\Phi_i=&-\td\rho^{-2}\left(\sum_{j=1}^m(\td{h}_{ij}-\td{H}\delta_{ij})e_j(\log {\td\rho})+e_i(\td{H})\right)\\
=&-\td\rho^{-2}\left(\td\rho\sum_{j=1}^K({h}_{ij}-\lambda\delta_{ij})e_j(\log {\td\rho})+e_i(\td\rho\lambda-\varepsilon \td n_0)\right)\\
=&-\td\rho^{-2}\left(\sum_{j=1}^K{h}_{ij}e_j(\td\rho)-\varepsilon e_i(\td n_0)\right)\\
=&-\td\rho^{-2}\left(\sum_{j=1}^Kh_{ij}\td y_0 E_j(\td y_0)-\td y_o E_i(\td n_0)\right)\\
=&-\td\rho^{-2}(h_i\td y_0 E_i(\td y_0)-\td y_0 h_iE_i(\td y_0))
=0.
\end{aligned}
\end{equation}

(2) $K+1\leq i\leq m$. We directly find
\begin{equation}
\begin{aligned}
\Phi_i=&-\td\rho^{-2}\left(\sum_{j=1}^m(\td{h}_{ij}-\td{H}\delta_{ij})e_j(\log {\td\rho})+e_i(\td{H})\right)\\
=&-\td\rho^{-2}\left(-\td\rho\lambda\sum_{j=K+1}^m\delta_{ij}e_j(\log {\td\rho})+e_i(\lambda\td\rho-\varepsilon\td n_0)\right)\\
=&\td\rho^{-2}\varepsilon e_i(\td n_0)=0.
\end{aligned}
\end{equation}

Therefore, $\Phi\equiv 0$.

\section{Proof of the main theorem}

To prove our main theorem, the following two theorems are needed:

\begin{thm}[cf. \cite{ag1}]\label{nie1} Let $x:M^m\to \mathbb{Q}^{m+1}_1$
be a regular space-like hypersurface with parallel conformal second
fundamental form. Then $x$ is locally conformal equivalent to one of the
following hypersurfaces:
\begin{enumerate}
\item $\mathbb{S}^{m-k}(a)\times\mathbb{H}^k\left(-\fr1{a^2-1}\right) \subset \mathbb{S}^{m+1}_1$, $a>1$,  $k=1,\cdots m-1$; or
\item $\mathbb{H}^k\left(-\fr1{a^2}\right)\times \bbr^{m-k} \subset \bbr^{m+1}_1$, $a>0$, $k=1,\cdots m-1$; or
\item $\mathbb{H}^k\left(-\fr1{a^2}\right)\times \mathbb{H}^{m-k}(-\fr1{1-a^2}) \subset \mathbb{H}^{m+1}_1$, $0<a<1$, $k=1,\cdots m-1$; or
\item $WP(p,q,a)\subset \bbr^{m+1}_1$ for some constants $p,q,a$, as indicated in
Example $3.1$.
\end{enumerate}
\end{thm}

\begin{thm}[cf. \cite{ag4}]\label{nie4}
Let $M^{m+1}_1(c)$ be a given Lorentzian space form of curvature $c$ and $x:M^m\rightarrow M^{m+1}_1(c)$ be a regular space-like hypersurface. If the conformal invariants of $x$ satisfy
\begin{equation}
\Phi\equiv0, \qquad A+\lambda B=\mu g \nnm
\end{equation}
for some smooth functions $\lambda,\mu$ on $M^m$, then both $\lambda$ and $\mu$ are constant, and $x$ is conformal equivalent to one of the space-like hypersurfaces in any of the three Lorentzian space forms which is of constant mean curvature and constant scalar curvature.
\end{thm}

Let $x:M^m\rightarrow \mathbb{S}^{m+1}_1$ be a regular space-like hypersurface, and suppose that the para-Blaschke tensor $D^\lambda$ is parallel. Since $D^\lambda$ is also symmetric, there exists a local orthonormal frame field $\{E_i\}$ around each point of $M^m$ such that
\begin{equation}
\label{Dij}
D^\lambda_{ij}=D^\lambda_i\delta_{ij}\mb{\ \ identically,}
\end{equation}
where $D^\lambda_i$'s are the eigenvalues of $D^{\lambda}$ and they are all constant. Since
\begin{equation}
0\equiv\sum
D^\lambda_{ijk}\omega^k=dD^\lambda_{ij}-D^\lambda_{kj}\omega^k_i-D^\lambda_{ik}\omega^k_j,
\end{equation}
we obtain that
\begin{equation}
\label{Dwij0}
 \omega^i_j=0 \quad \text{if } \ D^\lambda_i\neq D^\lambda_j.
\end{equation}

As the first step of the argument, we shall prove that the conformal form $\Phi\equiv 0$. For doing this, we need the following lemma:

\begin{lem}[cf. \cite{hl04}]\label{lem4.3} If $x$ has exactly two distinct conformal principal curvatures around a given point $p$ in $M^m$, then, around this point, the conformal second fundamental form $B$ of $x$ is parallel and the conformal form $\Phi\equiv 0$.
\end{lem}

\begin{proof} Let $b_1,b_2$ be the two distinct eigenvalues of $B$, which are necessarily constant by \eqref{trB}. Without loss of generality, we assume that there is some $K:\ 1\leq K\leq m-1$, such that the conformal principal curvatures $B_1=\cdots=B_K=b_1$, $B_{K+1}=\cdots=B_m=b_2$. On the other hand, the covariant derivatives $B_{ijk}$, $1\leq i,j,k \leq m$, are defined by
\begin{equation}\label{bijk}
\sum B_{ijk}\omega^k=dB_{ij}-\sum B_{kj}\omega^k_i-\sum B_{ik}\omega^k_j.
\end{equation}

Choose, around the give point $p\in M^m$, an orthonormal frame field $\{E_i\}$ with respect to the conformal metric $g$ such that $B_{ij}=B_i\delta_{ij}$ ($1\leq i,j\leq m$) identically. Denote
$$
I=\{i;\ 1\leq i\leq K\},\qquad J=\{i;\ K+1\leq i\leq m\}.
$$
If $i,j\in I$, then $$\sum B_{ijk}\omega^k=-b_1(\omega^j_i+\omega^i_j)=0.$$ It follows that $B_{ijk}=0$ for all $k=1,\cdots,m$.
Similarly, if $i,j \in J$, then $B_{ijk}=0$ for all $k=1,\cdots,m$.
By making use of the symmetry of $B$, we have $B_{ijk}=0$ for all $i,j,k$, that is,
the conformal second fundamental form $B$ is parallel.

Now from \eqref{Bijk}, it is easily derived that $\Phi\equiv 0$ around $p$.
\end{proof}

\begin{prop}\label{prop4.1} If the para-Blaschke tensor $D^{\lambda}$ of $x$ is parallel,
then the conformal form $\Phi\equiv 0$ on $M$.
\end{prop}

\begin{proof} If the proposition is not true, then there exists some point $p\in M^m$ such that $\Phi\neq 0$ at $p$ and thus around $p$. Choose an orthonormal frame field $\{E_i\}$ around $p$, such that $B_{ij}(p)=B_i\delta_{ij}$. By the assumption, there exists some $i_0$ such that $\Phi_{i_0}(p)\neq 0$. Then $\Phi_{i_0}\neq 0$ around the point $p$. Since $D^\lambda$ is parallel, we derive from \eqref{Dijk} that
\begin{equation}
\label{BPhi}
(B_i+\lambda)(\delta_{ij}\Phi_{i_0}-\delta_{ii_0}\Phi_j)=0.
\end{equation}

For any $i\neq i_0$, put $j=i$ in \eqref{BPhi}. It then follows that $B_i(p)+\lambda=0$ for each $i\neq i_0$. This proves that, around the point $p$, $x$ has exactly two distinct conformal principal curvatures with one of which being simple. It follows from Lemma \ref{lem4.3} that $\Phi\equiv 0$ around $p$, contradicting to the assumption that $\Phi(p)\neq 0$.
\end{proof}

In what follows, we use the orthonormal frame field $\{E_i\}$ such that \eqref{Dij} holds.

\begin{lem}\label{lemBij}  If $D^\lambda$ is parallel, then $B_{ij}=0$ whenever $D^\lambda_i\neq D^\lambda_j$.
\end{lem}

\begin{proof} By Proposition \ref{prop4.1}, $\Phi\equiv 0$. Then from \eqref{Phiij}, it follows that
$$
\sum
B_{ik}D^{\lambda}_{kj}-D^{\lambda}_{ik}B_{kj}=\Phi_{ij}-\Phi_{ji}=0,\quad \forall i,j.
$$
Thus, by \eqref{Dij}, for all $i,j$, $B_{ij}(D^\lambda_j-D^\lambda_i)=0$. It follows that
$B_{ij}=0$ whenever $D^\lambda_i\neq D^\lambda_j$.
\end{proof}

By Lemma \ref{lemBij}, we can choose around any point $p\in M^m$ a local orthonormal frame field $\{E_i\}$ such that both $D^\lambda$ and $B$ are
diagonalized simultaneously.

Let $t$ be the number of  distinct eigenvalues of $D^\lambda$, and $d_1,\cdots,d_t$ be  the distinct eigenvalues of $D^\lambda$. Then under the frame field $\{E_i\}$ chosen above, we can write
\begin{equation}
\label{DiaD}
(D^\lambda_{ij})=\mb{Diag}(\undbc{d_1, \cdots, d_1}_{k_1},
\undbc{d_2, \cdots, d_2}_{k_2}, \cdots, \undbc{d_t, \cdots, d_t}_{k_t}),
\end{equation}
namely,
\begin{equation}
\label{Ddt}
D^\lambda_1=\cdots=D^\lambda_{k_1}=d_1, \cdots,
D^\lambda_{m-k_t+1}=\cdots=D^\lambda_m=d_t.
\end{equation}

\begin{lem}\label{lemBD} If $D^{\lambda}$ is parallel  and $t\geq 3$, then, $B_i=B_j$ whenever $D^\lambda_i=D^\lambda_j$.
\end{lem}
\begin{proof} Under the local
orthonormal frame field $\{E_i\}$ chosen above, both \eqref{DiaD} and $B_{ij}=B_i\delta_{ij}$ hold.
 By \eqref{Dwij0}, for any $i,j$,  $\omega^i_j\equiv 0$ whenever  $D^\lambda_i\neq D^\lambda_j$.  Thus
$d\omega^i_j\equiv 0$ implying
\begin{equation}
0=B_{ij}^2-B_{ii}B_{jj} +(D^\lambda_{ii}-\lambda
B_{ii})-(D^\lambda_{ij}-\lambda B_{ij})\delta_{ij}
+(D^\lambda_{jj}-\lambda B_{jj})-(D^\lambda_{ij}-\lambda
B_{ij})\delta_{ij}.
\end{equation}
namely,
\begin{equation}
\label{DBiBj}
-B_iB_j-\lambda(B_i+B_j) +D^\lambda_i+D^\lambda_j=0.
\end{equation}
If there exist $i, j$ such that $D^\lambda_i=D^\lambda_j$ and
$B_i\neq B_j$, then for all $k$ satisfying $D^\lambda_k\neq D^\lambda_i$, we have
\begin{equation}
\label{BiBk}
-B_iB_k-\lambda(B_i+B_k) +D^\lambda_i+D^\lambda_k=0,\quad
-B_jB_k-\lambda(B_j+B_k) +D^\lambda_j+D^\lambda_k=0.
\end{equation}
From \eqref{BiBk}, we have $(B_j-B_i)(B_k+\lambda)=0$ which implies
$B_k=-\lambda$. Thus by \eqref{BiBk},
$$D^\lambda_k+\lambda^2=-D^\lambda_i=-D^\lambda_j.$$
This means that
$t=2$. This contradiction finishes the proof.
\end{proof}

Summing up, we have
\begin{cor}\label{corDD} Under the assumptions of Lemma $\ref{lemBD}$,  there
exists an orthonormal frame field $\{E_i\}$ such that
\begin{equation}
\label{DDDB}
D^\lambda_{ij}=D^\lambda_i\delta_{ij},\quad
B_{ij}=B_i\delta_{ij}
\end{equation}
and
\begin{align}
(D^\lambda_{ij})=&\mb{\rm Diag}(\undbc{d_1, \cdots, d_1}_{k_1},
\undbc{d_2, \cdots, d_2}_{k_2}, \cdots, \undbc{d_t, \cdots, d_t}_{k_t}),\nnm\\
(B_{ij})=&\mb{\rm Diag}(\undbc{b_1,\cdots, b_1}_{k_1},
\undbc{b_2, \cdots, b_2}_{k_2}, \cdots, \undbc{b_t, \cdots,
b_t}_{k_t}),\label{DDiaB}
\end{align}
where $b_1, \cdots, b_t$ are not necessarily different from each other.
\end{cor}

\begin{lem}\label{lemBC} Under the assumptions of Lemma $\ref{lemBD}$, all the conformal principal curvatures $b_1,\cdots,b_t$ of $x$ are constant, namely, $x$ is conformal isoparametric.
\end{lem}
\begin{proof} Without loss of generality, it suffices to show that $b_1$ is constant.
By the assumption and Corollary \ref{corDD}, we can choose a frame field $\{E_i\}$ in a neighborhood  of any point such that \eqref{DiaD}, \eqref{DDDB} and \eqref{DDiaB} hold. Note that for $1\leq i\leq k_1$ and
$j\geq k_1+1$, by \eqref{Dwij0},
\begin{equation}
\label{inoj}
\sum B_{ijk}\omega^k=dB_{ij}-\sum B_{kj}\omega^k_i-\sum
B_{ik}\omega^k_j=0.
\end{equation}
Therefore, $B_{ijk}=0$ for all $k$. By the symmetry of $B_{ijk}$, we see that
$B_{ijk}=0$, in case that any two of $i,j,k$ are less than or equal to
$k_1$ with the other  larger than $k_1$, or any one of  $i,j,k$ is less
than or equal to $k_1$ with the other two larger than $k_1$. Hence,
for any $i,j$ satisfying $1\leq i,j\leq k_1$,
\begin{equation}
\label{ieqj}
\sum_{k=1}^{k_1} B_{ijk}\omega^k=dB_{ij}-\sum
B_{kj}\omega^k_i-\sum
B_{ik}\omega^k_j=dB_i\delta_{ij}-B_j\omega^j_i-B_i\omega^i_j.
\end{equation}
We infer
\begin{equation}
\sum_{k=1}^{k_1} B_{iik}\omega^k=db_1,
\end{equation}
which yields
\begin{equation}
\label{Ekb}
E_k(b_1)=0,\quad k_1+1\leq k\leq m.
\end{equation}
Similarly,
\begin{equation}
E_i(B_j)=0,\quad 1\leq i\leq k_1,\ k_1+1\leq j\leq m.
\end{equation}

On the other hand, from \eqref{DBiBj} we have
\begin{equation}
-b_1B_j-\lambda(b_1+B_j)+d_1+D^\lambda_j=0, \quad k_1+1\leq j\leq m.
\end{equation}
We derive, for $1\leq k\leq k_1$,
\begin{equation}
\label{Ekb1}
E_k(b_1)(B_j+\lambda)=0,\quad 1\leq k\leq k_1,\ k_1+1\leq j\leq m.
\end{equation}

Define $U=\{q\in M^m; B_j(q)\neq-\lambda\ \text{for some\ }j\geq k_1+1\}$.
For any point $p\in U$, we can find some $j\geq k_1+1$ such that
$B_j\neq -\lambda$ around $p$. Therefore by \eqref{Ekb1}, $E_k(b_1)=0$ for
$1\leq k\leq k_1$ which with \eqref{Ekb} implies that $b_1$ is a constant.
This proves that $b_1$ is constant on the closure $\ol U$
of $U$.

On the other hand, for any $p\not\in\ol U$, we have
$B_{k_1+1}=\cdots=B_m=-\lambda$ around $p$. By \eqref{DDiaB}, $B$ has
exactly two distinct eigenvalues at $p$. Thus, if $M^m\bsl \ol U$ is a nonempty set,
from \eqref{trB}, we know that $b_1$ is a constant in  $M^m\bsl \ol U$.
Since  $M^m$ is connected, we have  that $b_1$ is constant identically on $M^m$.
\end{proof}

\begin{cor}\label{corBP} Under the assumptions of Lemma $\ref{lemBD}$, $B$ is parallel and $t$ must be $3$.
\end{cor}
\begin{proof} From \eqref{inoj} and \eqref{ieqj}, we infer that $B$ is parallel.

If $t>3$, then there exist at least four indices $i_1,i_2,i_3,i_4$,
such that
$D^\lambda_{i_1},D^\lambda_{i_2},D^\lambda_{i_3},D^\lambda_{i_4}$
are distinct from each other. Then we have from \eqref{DBiBj} that
\begin{equation*}
\begin{aligned}
&-B_{i_1}B_{i_2}-\lambda(B_{i_1}+B_{i_2})
+D^\lambda_{i_1}+D^\lambda_{i_2}=0,\quad
-B_{i_3}B_{i_4}-\lambda(B_{i_3}+B_{i_4})+D^\lambda_{i_3}+D^\lambda_{i_4}=0,\\
&-B_{i_1}B_{i_3}-\lambda(B_{i_1}+B_{i_3})
+D^\lambda_{i_1}+D^\lambda_{i_3}=0,\quad
-B_{i_2}B_{i_4}-\lambda(B_{i_2}+B_{i_4})+D^\lambda_{i_2}+D^\lambda_{i_4}=0.
\end{aligned}
\end{equation*}
It follows that
$(D^\lambda_{i_1}-D^\lambda_{i_4})(D^\lambda_{i_2}-D^\lambda_{i_3})=0$.
It is a contradiction proving that $t$ must be $3$.
\end{proof}

\begin{lem}\label{lemBnop}If $D^\lambda$ is parallel and
$B$ is not parallel, then one of the following cases holds:

\begin{enumerate}

\item  $t=1$ and $D^\lambda$ is proportional to the metric $g$;

\item  $t=2$, $d_1+d_2=-\lambda^2$ and $B_i=-\lambda$ hold either for all
$1\leq i\leq k_1$, or for all $k_1+1\leq i\leq m$.
\end{enumerate}
\end{lem}

\begin{proof} From Corollary \ref{corBP} it follows that $t\leq 2$. So it suffices to only consider the case
that $t=2$.

For any point $p\in M^m$, we can find an orthonormal frame field $\{E_i\}$ such that \eqref{Dij} holds around $p$ and $B_{ij}(p)=B_i\delta_{ij}$.
By \eqref{Dwij0}
\begin{equation}
\label{Dwij1}
\omega^i_j=0,\quad 1\leq i\leq k_1,\quad k_1+1\leq j\leq m,
\end{equation}
hold. By making use of the same assertion as \eqref{DBiBj}, we have
\begin{equation}
\label{DBiBj1}
-B_iB_j-\lambda(B_i+B_j) +D^\lambda_i+D^\lambda_j=0,\quad 1\leq i\leq
k_1,\quad k_1+1\leq j\leq m.
\end{equation}
If there exist $i_0,j_0$ with $1\leq i_0\leq k_1$, $k_1+1\leq j_0\leq m$ such that $B_{i_0}\neq -\lambda$
and $B_{j_0}\neq -\lambda$, then they are different from $-\lambda$ around the point $p$. It follows that, for any $i$, $1\leq i\leq k_1$,
$$
-B_{i_0}B_{j_0}-\lambda(B_{i_0}+B_{j_0})+D^\lambda_{i_0}+D^\lambda_{j_0}=0,\quad
-B_iB_{j_0}-\lambda(B_i+B_{j_0})+D^\lambda_i+D^\lambda_{j_0}=0.
$$
Thus,  $(B_i-B_{i_0})(B_{j_0}+\lambda)=0$. We obtain
\begin{equation}
B_i=B_{i_0},\quad 1\leq i\leq k_1.
\end{equation}
Similarly,
\begin{equation}
B_j=B_{j_0},\quad k_1+1\leq j\leq m.
\end{equation}
It follows that there are exactly two distinct conformal principal curvatures around point $p$. From \eqref{trB} we know that conformal principal curvatures $B_i$ are constant. By Lemma \ref{lem4.3}, $B$ is parallel around point $p$ and, by the arbitrariness, $B$ is parallel everywhere. This is indeed a contradiction. Therefore it must holds that either $B_i=-\lambda$ for $i$, $1\leq i\leq k_1$, or $B_j=-\lambda$ for  $j$, $k_1+1\leq j\leq m$. Thus by \eqref{DBiBj1}, $d_1+d_2=-\lambda^2$.
\end{proof}

Now we are in a position to complete the proof of our main theorem (Theorem \ref{main}).

By Theorem \ref{nie1} and Theorem \ref{nie4}, to prove the main theorem, we need to show that if $x$ does not have parallel conformal second fundamental form and the number $t$ of the distinct eigenvalues of $D^\lambda$ is
larger than $1$, then $x$ must be locally conformal equivalent to one of the hypersurfaces given in Examples 3.2 and 3.3.

If $t\geq 3$, then by Corollary \ref{corBP}, the conformal second fundamental form is parallel. Hence, $t=2$. Without
loss of generality, we can assume, by Lemma \ref{lemBnop}, that
\begin{equation}
\label{d2B}
t=2,\quad d_1=d,\quad d_2=-\lambda^2-d,\quad
B_{K+1}=\cdots=B_m=-\lambda,
\end{equation}
where $K=k_1$. Since the conformal second fundamental form $B$ of $x$ is not parallel, the number of distinct conformal principal curvatures must be larger than $2$ (see Lemma \ref{lem4.3}). It follows easily that $m\geq 3$. Because $D^\lambda$ is parallel, the tangent bundle $TM^m$ has a decomposition $TM^m=V_1\oplus V_2$, where $V_1$ and
$V_2$ are eigenspaces of $D^\lambda$ corresponding to eigenvalues $d_1=d$ and $d_2=-\lambda^2-d$, respectively.

Let $\{E_i,1\leq i\leq K\}$, $\{E_j,K+1\leq j\leq m\}$ be orthonormal frame fields for subbundles $V_1$ and $V_2$, respectively. Then $\{E_i,1\leq i\leq m\}$ is an orthonormal frame field on $M^m$ with respect to the conformal metric $g$. On the other hand, Equation \eqref{Dwij1} implies that both $V_1$ and $V_2$ are integrable, and thus the Riemannian manifold $(M^m,g)$ can
be decomposed locally into a direct product of two Riemannian manifolds $(M_1,g_1)$ and $(M_2,g_2)$, that is,
\begin{equation}
(M,g)=(M_1,g_1)\times (M_2,g_2).
\end{equation}
It follows from \eqref{Rijkl} and \eqref{Ddt} that the Riemannian curvature tensors of $(M_1,g_1)$ and $(M_2,g_2)$ have respectively the following components:
\begin{align}
R_{ijkl}=&(2d+\lambda^2)
(\delta_{il}\delta_{jk}-\delta_{ik}\delta_{jl})
+(B_{jl}+\lambda\delta_{jl})(B_{ik}+\lambda\delta_{ik})\nnm\\
&\ \ -(B_{il}+\lambda\delta_{il})(B_{jk}+\lambda\delta_{jk}),\qquad 1\leq
i,j,k,l\leq K;\label{DM1R}\\
R_{ijkl}=&-(\lambda^2+2d)
(\delta_{il}\delta_{jk}-\delta_{ik}\delta_{jl}),\quad K+1\leq
i,j,k,l\leq m.\label{DM2R}
\end{align}
Thus $(M_2,g_2)$ is of constant sectional curvature $-(\lambda^2+2d)$. Since $d_1\neq d_2$, equation\eqref{d2B} implies that $-(2d+\lambda^2)\neq 0$.

Next, we consider separately the following two subcases:

Subcase (1): $(2d+\lambda^2)>0$. In this case, set $r=(2d+\lambda^2)^{-\fr12}$, then $(M_2,g_2)$ can be locally identified with $\mathbb{H}^{m-K}\left(-\fr1{r^2}\right)$. Let $$\td y=(\td y_0,\td y_2):\mathbb{H}^{m-K}\left(-\fr1{r^2}\right) \to\bbr^1_1\times\bbr^{m-K}\equiv\bbr^{m-K+1}_1$$ be the canonical embedding. Since
$$h=\sum_{i,j=1}^K(B_{ij}+\lambda\delta_{ij})\omega^i\omega^j$$ is a
Codazzi tensor on $(M_1,g_1)$, it follows from \eqref{DM1R} that there exists a hypersurface
\begin{equation}
\td y_1:(M_1,g_1)\to \mathbb{S}^{K+1}_1(r)\subset\bbr^{K+2}_1,\quad 2\leq K\leq m-1,
\end{equation}
such that
$h$ is its second fundamental form. Clearly, $\td y_1$ has constant scalar curvature $\td S_1$
and constant mean curvature $\td H_1$ as follows:
$$
\td S_1=\fr{mK(K-1)+(m-1)r^2}{mr^2}-m(m-1)\lambda^2,\quad \td H_1=\fr{m}{K}\lambda.
$$
Note that $M^m$ can be locally identified with $\td M^m=(M_1,g_1)\times \mathbb{H}^{m-K}(-\fr1{r^2})$.

Define $\td x_1={\td y_1}/{\td y_0}$, $\td x_2={\td y_2}/{\td y_0}$
and $\td x=(\td x_1,\td x_2)$. Then, by the discussion in Example
3.2, $\td x:\td M^m\to \mathbb{S}^{m+1}_1$ must be a regular space-like hypersurface with
the given $g$ and $B$ as its conformal metric and conformal second fundamental
form, respectively. Therefore, by Theorem \ref{thm2.4}, $x$ is locally conformal equivalent to $\td x$.

Subcase (2): $2d+\lambda^2<0$. In this case, set $r=(-(2d+\lambda^2))^{-\fr12}$,
then $(M_2,g_2)$ can be locally identified with $\mathbb{S}^{m-K}(r)$. Let
$\td y_2:\mathbb{S}^{m-K}(r)\to\bbr^{m-K+1}$ be the canonical embedding. Similarly as above, since $$h=\sum_{i,j=1}^K(B_{ij}+\lambda\delta_{ij})\omega^i\omega^j$$ is a Codazzi
tensor on $(M_1,g_1)$, it follows from \eqref{DM1R} that there exists a hypersurface
$$\td y=(\td y_0,\td y_1):(M_1,g_1)\to
\mathbb{H}^{K+1}_1\left(-\fr1{r^2}\right)\subset\bbr^1_1 \times\bbr^{K+1}_1\equiv\bbr^{K+2}_2,\quad 2\leq
K\leq m-1,$$
which has $h$ as its second fundamental form, and $\td y$ clearly has constant scalar curvature $\td S_1$ and constant mean curvature $\td H_1$ where
$$\td S_1=\fr{-mK(K-1)+(m-1)r^2}{mr^2}-m(m-1)\lambda^2,\quad \td H_1=\fr{m}{K}\lambda,$$
and $M^m$ can be locally identified with $\td M^m=(M_1,g_1)\times
\mathbb{S}^{m-K}\left(r\right)$.

Write $\td y_1=(\td y'_1,\td y''_1)\in \bbr^1_1\times\bbr^{K}\equiv\bbr^{K+1}_1$. Then $$\td y^2_0+\td y'_1{}^2=r^2+\td y''_1\cdot \td y''_1> 0.$$
Hence, without loss of generality, we can assume that $\td y_0\neq 0$. Define $\veps=\sgn(\td y_0)$ and let $\td x_1=\veps{\td y_1}/{\td y_0}$, $\td x_2=\veps{\td y_2}/{\td y_0}$
and $\td x=(\td x_1,\td x_2)$. Then, by the discussion in Example
3.3, $\td x:\td M^m\to \mathbb{S}^{m+1}_1$ is a regular space-like hypersurface
with the given $g$ and $B$ as its conformal metric and conformal second
fundamental form, respectively. Once again we use Theorem \ref{thm2.4} to assure that $x$ is locally conformal equivalent to $\td x$.\hfill $\Box$

\end{document}